\documentclass[a4paper, oneside, reqno, 12pt]{amsart}
\usepackage[utf8]{inputenc}
\usepackage{amsmath, amssymb, amsfonts}
\usepackage{fullpage}
\usepackage{textcomp, cmap}
\usepackage[nobysame, alphabetic]{amsrefs}
\usepackage[english]{babel}

\usepackage{pgfplots}

\DeclareMathOperator{\conv}{conv}

\DeclareMathOperator{\diam}{diam}

\usepackage{hyperref}
\hypersetup{
    colorlinks   = true, 
    urlcolor     = blue, 
    linkcolor    = red, 
    citecolor   = green,
    pdftitle= {No-dimensional Tverberg problem},   
    pdfauthor= {Alexander Polyanskii}
}

\usepackage{mathtools}

\usepackage{cleveref}
\theoremstyle{plain}
\newtheorem{theorem}{Theorem}
\crefname{theorem}{Theorem}{Theorems}
\newtheorem{proposition}[theorem]{Proposition}
\crefname{proposition}{Proposition}{Propositions}
\newtheorem{lemma}[theorem]{Lemma}
\crefname{lemma}{Lemma}{Lemmas}
\theoremstyle{definition}
\newtheorem{problem}[theorem]{Problem}
\crefname{problem}{Problem}{Problems}
\newtheorem{openproblem}[theorem]{Open Problem}
\crefname{openproblem}{Open Problem}{Open Problems}
\theoremstyle{remark}
\newtheorem{remark}[theorem]{Remark}
\crefname{remark}{Remark}{Remarks}
\newtheorem{example}[theorem]{Example}
\crefname{example}{Example}{Examples}

\crefname{section}{Section}{Sections}

\title[No-dimensional Tverberg-type problems]{No-dimensional Tverberg-type problems}
\author[A.~Polyanskii]{{Alexander~Polyanskii}}

\address{Alexander Polyanskii,
\newline\hphantom{iii} Department of Mathematics, Emory University, Atlanta, GA, 30322
}
\email{\href{mailto: apolian@emory.edu}{apolian@emory.edu}}
\urladdr{\url{http://polyanskii.com}
}

\pgfplotsset{compat=1.18}
\begin{document}
\begin{abstract}
       Recently, Adiprasito \textit{et al.} have initiated the study of the so-called no-dimensional Tverberg problem. This problem can be informally stated as follows: \textit{Given $n\geq k$, partition an $n$-point set in Euclidean space into $k$ parts such that their convex hulls intersect a ball of relatively small radius.}

   \sloppy In this survey, we aim to present the recent progress towards solving the no-dimensional Tverberg problem and new open questions arising in its context. Also, we discuss the colorful variation of this problem and its algorithmic aspects, particularly focusing on the case when each part of a partition contains exactly 2 points. The latter turns out to be related to the following no-dimensional Tverberg-type problem of Huemer \textit{et al.}: \textit{For an even set of points in Euclidean space, find a perfect matching such that the balls with diameters induced by its edges intersect.} 
\end{abstract}
\maketitle

\section{Introduction}
\label{introduction}
In 1966, Helge Tverberg published his seminal work~\cite{Tverberg1966}, extending Radon's theorem~\cite{Radon1921mengen} on the intersection of two convex hulls. Tverberg's theorem became a cornerstone for many geometric and topological questions, constituting a vast direction at the crossroads between discrete geometry and topological combinatorics.

\begin{quote} \textbf{Tverberg's theorem.}
\textit{For any set of $(k-1)(d+1)+1$ points in $\mathbb{R}^d$, there is a partition into $k$ parts such that their convex hulls intersect.} 
\end{quote}

A simple dimension-counting argument shows that the number $(k-1)(d+1)+1$ in Tverberg's theorem is minimal to ensure the intersection property. This observation raises the following questions: \textit{What happens if we have a smaller point set? Can such a set be partitioned into $k$ parts such that their convex hulls are arranged relatively close to each other?} 

In a recent foundational paper~\cite{adiprasito2020theorems}, Adiprasito~\textit{et al.} addressed these questions. In particular, they explored the so-called no-dimensional Tverberg problem, which informally asks: \textit{Given integers $n\geq k$, partition an $n$-point set in Euclidean space into $k$ parts such that their convex hulls intersect a ball of relatively small radius with respect to the diameter of the set}. For the formal statement, see \cref{problem standard}. By considering the vertex set of a regular unit $(n-1)$-dimensional simplex, one can compute that partitioning into $k$ parts of nearly equal size minimizes the radius of a ball intersecting the convex hulls of these parts, with the radius equal to $(1-o(1))\sqrt{k/(2n)}$\footnote{Here, $o(1)$ is a function of $n$ and $k$ that approaches $0$ as $k\to \infty$.}.

Given this example, Adiprasito \textit{et al.} establish an upper bound for the minimum radius in the no-dimensional Tverberg problem, which is tight up to a multiplicative constant; see  \cite{adiprasito2020theorems}*{Theorem~1.3}.

\begin{theorem}[No-dimensional Tverberg theorem]
\label{adiprasito theorem}    
    For integers $n\geq k\geq 2$ and an $n$-point set of diameter $D$ in Euclidean space, there exist a partition of the set into $k$ subsets and a ball of radius 
\[
    (2 + \sqrt{2}) \sqrt{\frac{k}{n}}\cdot D
\]
intersecting the convex hull of each subset.
\end{theorem}

The lack of dependence of the statement on the specific dimension motivates the name of this theorem. Essentially, this result holds in the real Hilbert space $\ell_2$.

It turned out that the proof of \cref{adiprasito theorem} reduces to showing the existence of a ball that not only intersects the convex hulls of all subsets but also contains particular points of these convex hulls --- their centroids. This idea, in one way or another, is used in all results that generalize \cref{adiprasito theorem}. Thus, it is natural to consider the following refinement of the no-dimensional Tverberg problem: \textit{Given integers $n \geq k$, partition an $n$-point set in $\ell_2$ into $k$ parts such that their centroids lie within a ball of relatively small radius with respect to the diameter of the set.}

This problem can be approached from a slightly different perspective, namely by attempting to prove that certain small balls centered at the centroids of these $k$ parts intersect. If these balls are assumed to be congruent, we are essentially considering an equivalent problem. However, if we start varying the radii depending on the points in each part, a new Tverberg-type problem emerges. As will be discussed at the end of \cref{section euclidean}, this very observation was employed in~\cite{barabanshchikova2024intersectingballs} to completely resolve one of the colorful variation of the no-dimensional Tverberg problem; see \cref{problem colorful euclidean new}. 

If additionally we assume that $n=2k$ and each part contains exactly 2 points, then the latter version of the problem has an interesting choice of radii, first explored by Huemer~\textit{et al.}~\cite{huemer2019matching}: The \textit{diametral ball} induced by a pair $\{a,b\}$ of points in $\ell_2$ is the ball whose diameter is the line segment $ab$. In particular, they studied the following problem: \textit{For an even set of points, find a perfect matching such that the diametral balls induced by the edges intersect.} For a more general variation of the problem, see \cref{section k=2}. At nearly the same time as Adiprasito \textit{et al.}, Huemer~\textit{et al.} proved the following result.

\begin{theorem}
\label{red-blue matching}
For any $n$ red and $n$ blue points in the plane, there is a perfect red-blue matching such that the diametral disks induced by the edges of the matching intersect.
\end{theorem}

Later, it was shown~\cite{Pirahmad2022}*{Theorem 1.4} that this theorem holds in $\ell_2$, that is, does not depend on the specific dimension of the space as well. So one can consider this generalization of \cref{red-blue matching} as another no-dimensional Tverberg theorem.

\smallskip

The main aim of this survey is to compile known results related to generalizations of these two problems, arising in the works of Adiprasito \textit{et al.}~\cite{adiprasito2020theorems} and Huemer \textit{et al.}~\cite{huemer2019matching}.

In this survey, we discuss a wide range of questions, problems, and conjectures. To avoid confusion in the names, throughout the survey, we refer to problems including conjectures, which are still open in their generality, as \textit{Open Problems} and completely resolved ones simply as \textit{Problems}.

\subsubsection*{Structure of the survey.} In the next section, we briefly review algorithmic versions of Tverberg's theorem, some of its proofs, and the statement of its colorful version. We include only those facts that are highly relevant to other parts of this survey. \cref{section euclidean} discusses the no-dimensional Tverberg problem by Adiprasito \textit{et al.} and its colorful version in $\ell_2$, including the algorithmic aspects of the problem. \cref{section k=2} focuses on the intersecting properties of diametral balls, extending \cref{red-blue matching}. \cref{section banach} covers the known results generalizing the no-dimensional Tverberg theorem for the case of Banach spaces. In \cref{section: applications}, we discuss applications of the no-dimensional Tverberg results. In \cref{section metric spaces}, we propose a very general version of the problem for metric spaces.

\subsubsection*{Notation} Throughout the paper, we use the standard notation $[n]$ for the set $\{1,\dots, n\}$. Also, $\conv P,$ $c(P),$ and $\diam P$ stand for the convex hull, centroid, and diameter of a finite point set $P$. Recall that the \textit{centroid} of a finite point set $P$ is defined by
$
    c(P):=\frac{1}{|P|}\sum_{p\in P} p,
$
where $|P|$ stands for the cardinality of $P$. 

% For a family $\mathcal Q=\{Q_1,\dots, Q_r\}$ of point sets in a metric space $(M,d)$, we define two invariants $\diam_1 \mathcal Q$ and $\diam_2 \mathcal Q$ as follows
% \[
% \diam_1 \mathcal Q:=\max_{j\in [r]} \diam Q_j
% \text{ and }
% \diam_2 \mathcal  Q:= \max_{\substack{x\in Q_i, y\in Q_j\\ i\ne j}} d(x,y).
% \]
% The triangle inequality easily implies that $\diam_1 \mathcal Q\leq 2 \diam_2 \mathcal Q$. One can easily verify that there is no constant $c$ such that $\diam_2 \mathcal Q\leq c \diam_1 \mathcal Q$.

\section{Remarks on Tverberg's theorem}
\label{section discussion of Tverberg theorem its proofs and variations}

\label{three remarks} 
In this section, we present three independent remarks about Tverberg's theorem that are closely related to the rest of the paper. For a comprehensive recent survey on Tverberg’s theorem and its generalization, we refer the reader to the work~\cite{Baranysoberon2018tverberg} of B\'ar\'any and Sober\'on.

First, there is no known polynomial-time \textit{algorithm} to compute a partition satisfying Tverberg's theorem; see Open problem 3.15 and the surrounding paragraphs in~\cite{DeLoera2019discrete}. This naturally leads to the question of whether anyone can effectively solve appropriate versions of Tverberg's theorem (see also~\cite{Baranysoberon2018tverberg}*{Section 4.6}), including the no-dimensional Tverberg problems.

Second, Tverberg's theorem can be proven using various \textit{methods}, and among these, we highlight two that have been used to study the no-dimensional Tverberg problem. The first approach relies on the so-called Sarkaria's tensor trick~\cites{sarkaria1992tverberg, barany1997colourful}. The idea is to embed several copies of the original point set in a higher-dimensional space and then apply B\'ar\'any's celebrated colorful Carath\'eodory theorem~\cite{barany1982generalization}. Another approach involves a well-chosen function that depends on a partition and a point (which is expected to be common for the convex hulls in Tverberg's theorem), such that its extremum is attained at a partition and a point satisfying Tverberg's theorem. This method was initially implicitly used by Tverberg~\cite{tverberg1981generalization} and later developed by Tverberg and Vrećica~\cite{Tverberg1993}; an elegant modification of the latter proof was proposed by Roudneff~\cite{Roudneff2001}. We refer the interested reader to the brief exposition of these approaches provided in~\cite{Baranysoberon2018tverberg}*{Section~1.2}.

\sloppy
Lastly, we should mention the \textit{colorful Tverberg conjecture} by B\'ar\'any and Larman~\cite{Barany1992}, the general case of which remains one of the most appealing open problems in the field:

\begin{quote} 
\textbf{Colorful Tverberg conjecture.}
\textit{For pairwise disjoint sets $Q_1,\dots, Q_{d+1}$, each consisting of $k$ points in $\mathbb R^d$, there exists a partition $\{P_1,\dots, P_k\}$ of the union $Q_1\cup \dots \cup Q_{d+1}$ such that each~$P_i$ has exactly one point of each $Q_j$ and the convex hulls of $P_1, \dots, P_k$ intersect.}
\end{quote}

\section{No-dimensional Tverberg problems in \texorpdfstring{$\ell_2$}{l2}}
\label{section euclidean}

\subsection{No-dimensional Tverberg problem and its colorful variation} Recently, Adiprasito \textit{et al.}~\cite{adiprasito2020theorems} considered the following problem which remains open. 

\begin{openproblem}[no-dimensional Tverberg problem]
 \label{problem standard}
    Given an integer $k \geq 2$ and a real $r\geq 1$ such that $n=kr$ is an integer, find the smallest value of $R:=R(\ell_2,k,r)$ such that for any $n$-point set $P$ in $\ell_2$, there are a partition $\{P_1,\dots, P_k\}$ of $P$ and a ball of radius $R\cdot \diam P$ intersecting $\conv P_i$ for all $i\in [k]$.
\end{openproblem}

For this problem, we use the notation $r=n/k$ to represent the average number of points in each subset of the partition. For clarity, one may assume that $r$ is an integer. 

Jung's theorem~\cite{jung1901} on the minimum enclosing ball of a bounded set in $\mathbb R^d$ provides a solution to~\cref{problem standard} for $r=1$:
\[
    R(\ell_2,k,1)= \sqrt{\frac{k-1}{2k}}.
\]

Adiprasito \textit{et al.}~\cite{adiprasito2020theorems}*{the last paragraph of Section~6 and Theorem~1.3} proved that 
\[
    \frac{1-o(1)}{\sqrt{2r}}
    \leq R(\ell_2,k,r)\leq\frac{2+\sqrt{2}}{\sqrt{r}}.
\]
As we discussed in the introduction, the lower bound follows from considering $P$ as the vertex set of a regular simplex, while the upper bound is a consequence of \cref{adiprasito theorem}. This theorem can be directly derived from their corresponding result on the colorful variation of \cref{problem standard}.

\begin{openproblem}[colorful no-dimensional Tverberg problem]
\label{problem colorful euclidean}
Given integers $k\geq 2$ and $r\geq 1$, find the smallest value of $R_c:=R_c(\ell_2, k,r)$ such that for any pairwise disjoint sets $Q_1,\dots, Q_r$ in $\mathbb \ell_2$, each of size $k$, there is a partition $\{P_1,\dots, P_k\}$ of the union $P=Q_1\cup\dots \cup Q_r$, with $|P_i\cap Q_j|=1$ for any $i\in [k]$ and $j\in [r]$, and a ball of radius $R_c \cdot \max_{j\in [r]} \diam Q_j$ intersecting $\conv P_i$ for all $i\in [k]$.
\end{openproblem}

For consistency between Open Problems \ref{problem standard} and \ref{problem colorful euclidean}, we denote the size of $P$ by $n=kr$ in the latter problem. 

The case $r=1$ in \cref{problem colorful euclidean} also follows from Jung's theorem:
\[
R_c(\ell_2,k,1)=\sqrt{\frac{k-1}{2k}}.
\]

\begin{example}
    \label{example simplex}
    To prove the same lower bound for \cref{problem colorful euclidean}, one can choose each $Q_i$ as a set of unit vectors pointing to vertices of a regular $(k-1)$-simplex centered at the origin such that the linear spans of $Q_i$ are pairwise orthogonal. The sets $Q_1,\dots, Q_r$ can be easily shown to yield the following lower bound:
    \[
        \frac{1-o(1)}{\sqrt{2r}}
        \leq R_c(\ell_2, k,r),
    \]
    where the function $o(1)$ depends only on $k$.
\end{example}

Originally, Adiprasito \textit{et al.}~\cite{adiprasito2020theorems}*{Theorem~6} studied \cref{problem colorful euclidean} under the slighly different condition that the ball has radius $R_c \cdot \diam P$ instead of $R_c\cdot \max_{j\in [r]} \diam Q_j$. The current version of \cref{problem colorful euclidean} was proposed by Choudhary and Mulzer in~\cite{Choudhary2022}, and later its extension for Banach spaces was also considered by Ivanov in~\cites{Ivanov2021}; see \cref{problem colorful banach}. 

Although Adiprasito \textit{et al.}~\cite{adiprasito2020theorems} stated their result in slightly different terms, their proof implies the following bounds for $R$ in the above-mentioned formulation of \cref{problem colorful euclidean}:
\begin{equation}
\label{equation colorful adiprasito}
R_c(\ell_2,k,r)\leq \frac{1+\sqrt{2}}{\sqrt{r}}.
\end{equation}

Adiprasito \textit{et al.} obtained the upper bound in \cref{problem standard} by noting that if $k\geq 2$ is an integer and $r\geq 1$ is a real number such that $n = kr$ is an integer, then $r < 2\lfloor r \rfloor$ and 
\[
R(\ell_2, k, r) \leq R(\ell_2, k, \lfloor r \rfloor) \leq R_c(\ell_2, k, \lfloor r \rfloor).
\]

To prove~\eqref{equation colorful adiprasito}, Adiprasito \textit{et al.} successively partition the sets $Q_1,\dots, Q_r$ by using \cref{lemma adiprasito}, which we state below only for even sets. According to their proof, the desired ball is centered at $c(P)$ and contains $c(P_i)\in \conv P_i$ for each $i\in [k]$. To bound the distance between $c(P)$ and $c(P_i)$, they apply the triangle inequality\footnote{Let us note the indirectly related work by Simon~\cite{simon2016average}, which addresses a topological version of Tverberg's theorem on the coincidence of centroids of images of points chosen on disjoint faces of a simplex.}. 
\begin{lemma} 
\label{lemma adiprasito}
    For pairwise disjoint $2m$-point sets $Q_1,\dots, Q_{r}$ in $\ell_2$, there is a subset \( Q \) of the union \( P:=Q_1\cup \dots \cup Q_r \), sharing exactly $m$ points with $Q_i$ for each $i\in [r]$, such that
\[
    \|c(Q)-c(P)\| \leq \frac{1}{\sqrt{2m-1}}\cdot \frac{\diam P}{\sqrt{2r}}.
\]
\end{lemma}

The proof of this lemma relies on an averaging approach, but an effective derandomization of their proof is not known.

\begin{remark}
    It is worth noting that a similar averaging technique for centroids was also applied to study a no-dimensional version of Carath\'eodory's theorem, a closely related result to Tverberg's theorem. In particular, Adiprasito \textit{et al.}~\cite{adiprasito2020theorems}*{Theorem~2.1} proved the following result (see also \cref{remark2}):

\begin{theorem}[colorful no-dimensional Carath\'eodory theorem]
\label{colorful no-dimensional caratheodory theorem}
 Let $P_1, \dots, P_r$ be $r \geq 2$ pairwise disjoint point sets in $\ell_2$ such that a point $o$ lies in $\conv P_i$ for each $i\in [r]$. Then there exists $Q\subset P_1\cup \dots \cup P_r$ with $|Q\cap P_i|=1$ such that the ball centered at $o$ of radius
\[
     \frac{1}{\sqrt{2r}}\max_{i\in [r]} \diam P_i
\]  
hits $\conv Q$.
\label{caratheodory}
\end{theorem}
If we additionally assume that the points lie in $\mathbb R^d$ instead of $\ell_2$, then the proof of~\cref{colorful no-dimensional caratheodory theorem} admits a linear-time $O(nd)$ derandomization. Additionally, using the Frank--Wolfe algorithm~\cite{frank1956algorithm}, they obtain a generalization of \cref{caratheodory}, which resembles a strengthening of the colorful Carath\'eodory theorem from~\cites{arocha2009very, holmsen2008points}. We refer the interested reader to Sections 2 and 3 in~\cite{adiprasito2020theorems}, where the authors thoroughly review related previous results.
\end{remark}

\subsection{Algorithmic aspects} Motivated by results of Adiprasito \textit{et al.}, Choudhary and Mulzer~\cite{Choudhary2022} initiated the study of algorithmic aspects of the no-dimensional Tverberg problems. Assuming additionally that points of $P$ lie in $\mathbb R^d$ instead of $\ell_2$, they found algorithms with running time $O(nd \log k)$ for \cref{problem standard} and $O(ndk)$ for \cref{problem colorful euclidean}. 
The side effect of their approach is that their bounds on the radius are slightly worse, by a factor of $O(\sqrt{k})$. Specifically, they showed the bound
\[
R(\ell_2, k,r)\leq O(\sqrt{k/r}) \text{ and }R_c(\ell_2,k,r)\leq O(\sqrt{k/r}).
\]

Their proof is based on Sarkaria's tensor trick; see the second remark in \cref{section discussion of Tverberg theorem its proofs and variations}. By applying this trick, Choudhary and Mulzer~\cite{Choudhary2022} reduced Open Problems~\ref{problem standard} and~\ref{problem colorful euclidean} to a problem in $\mathbb R^{dk}$ and used a technical no-dimensional lemma resembling the colorful Carathéodory theorem. The lemma was also proved using an averaging approach. This reduction, on the one hand, allows for efficient derandomization, resulting in polynomial algorithms. On the other hand, it leads to weaker bounds on radii.

In 2023, Har-Peled and Robson~\cite{harpeled2023} found a polynomial algorithm for \cref{problem standard}. Their argument essentially relies on a variation of \cref{lemma adiprasito} for $r=1$. However, their proof can be implemented with a simple deterministic algorithm.  Similarly to the approach by Adiprasito \textit{et al.}~\cite{adiprasito2020theorems}, they successively applied the triangle inequality and this lemma to compute the partition that gives the upper bound
\[
    R(\ell_2,k,r)\leq \sqrt{8/r}
\]
in a running time of $O(nd\log n)$ if the points of $P$ lie in $\mathbb R^d$. Also, they provided a probabilistic algorithm (with a polynomial derandomization), which gives a better upper bound
\[
    R(\ell_2,k,r)\leq \sqrt{2/r}.
\]

\subsection{Second colorful no-dimensional Tverberg problem} 
Very recently, a new colorful no-dimensional Tverberg problem has been studied in~\cite{BarabanshchikovaPolyanskii2024Tverberg}. Note that the only subtle difference between \cref{problem colorful euclidean} and the new problem lies in replacing $\max_{j \in [r]} \diam Q_j$ with the maximum distance between points from distinct parts $Q_j$, that is, 
\[
\diam (Q_1,\dots, Q_r):=\max \{ \|x-y\|  \mid x\in Q_i,y\in Q_j \text{ with } i\ne j\}.
\]

\begin{problem}[second colorful no-dimensional Tverberg problem]
\label{problem colorful euclidean new}
Given integers $k\geq 2$ and $r\geq 2$, find the smallest value of $R_c':=R_c'(\ell_2, k,r)$ such that for pairwise disjoint sets $Q_1,\dots, Q_r$ in $\ell_2$, each of size $k$, there is a partition $\{P_1,\dots, P_k\}$ of the union $P=Q_1\cup\dots \cup Q_r$, with $|P_i\cap Q_j|=1$ for any $i\in [k]$ and $j\in [r]$, and a ball of radius 
\(
    R'_c \cdot \diam (Q_1,\dots, Q_r)
\)
intersecting $\conv P_i$ for all $i\in [k]$.
\end{problem}

Interestingly, there is a colorful version of Jung's theorem proved by Akopyan~\cite{akopyan2013combinatorial}*{Theorem~5}, which involves $\diam(Q_1,\dots, Q_r)$. Thus, exploring connections between \cref{problem colorful euclidean new} and the following result would be worthwhile.
\begin{theorem}[colorful no-dimensional Jung's theorem] 
\label{theorem akopyan}
For finite sets $Q_1,\dots, Q_r$ in $\ell_2$, there is a ball of radius $\diam (Q_1,\dots, Q_r)/\sqrt{2}$ covering one of the sets $Q_1,\dots, Q_r$.
\end{theorem}

\begin{remark}
\label{remark connecting old colorful and new one}
    For finite sets $Q_1,\dots,Q_r$ in $\ell_2$, by triangle inequality, we have
    \[
        \max_{j\in [r]} \diam Q_j\leq 2 \diam (Q_1,\dots, Q_r).
    \]
    Hence, we have that $R_c' (\ell_2,k,r)\leq 2R_c (\ell_2,k,r)$. Since there is no $\alpha$ such that
    \[
            \diam (Q_1,\dots, Q_r)\leq \alpha \max_{j\in [r]} \diam Q_j,
    \]
    we cannot easily derive that  $R_c (\ell_2,k,r)\leq \alpha R_c' (\ell_2,k,r)$
    Analogously, for integers $k\geq 2$ and $r\geq 2$ one can show that $R(\ell_2,k,r)\leq R_c'(\ell_2,k,r)$ and cannot easily derive that $R_c' (\ell_2,k,r)\leq \alpha R (\ell_2,k,r)$.
\end{remark}

Given \cref{remark connecting old colorful and new one}, the reader may wonder about the necessity of studying \cref{problem colorful euclidean new}. There are two reasons for this. First, unlike Open Problems~\ref{problem standard} and~\ref{problem colorful euclidean}, \cref{problem colorful euclidean new} has a precise solution. Second, its solution provides a new approach to derandomizing no-dimensional results, which may inspire new techniques for eliminating randomness in arguments for other similar problems.

Barabanshchikova and the author~\cite{BarabanshchikovaPolyanskii2024Tverberg}*{Theorem~1} have completely solved \cref{problem colorful euclidean new} by proving the tight upper bound
\begin{equation}
    \label{equation colored upper bound}
    R_c'(\ell_2,k,r)\leq \frac{1}{\sqrt{2r}}.
\end{equation}
The tightness of this bound can be demonstrated using \cref{example simplex}.

The idea of their proof is very different from the averaging technique used to provide the upper bound for $R(\ell_2, k,r)$ by Adiprasito \textit{et al.} and in subsequent papers. For that, Barabanshchikova and the author adapted ideas from~\cite{Pirahmad2022}*{Theorem~1.4}, where the case $r=2$ was studied. 
In particular, they used the method for proving Tverberg's theorem through optimization; see the second remark in \cref{three remarks}. However, their approach introduces a subtle yet important innovation, which we discuss in the paragraph after the sketch of the solution to \cref{problem red-blue no-dimensional} in \cref{subsection red-blue tverberg matchings}.

Next, we limit ourselves by a brief description of the proof of the general case. Barabanshchikova and the author~\cite{BarabanshchikovaPolyanskii2024Tverberg} consider a function depending on a partition $\{P_1,\dots, P_k\}$, which returns the sum of squared distances between pairs of points from the same subset $P_i$ of the partition. It turns out that this function attains its maximum at the partition such that the balls centered at $c(P_i)$ of radius $\diam P_i/\sqrt{2r}$ intersect. One can easily verify that the ball centered at this common point of radius  $\max_{i\in [k]} \diam P_i/\sqrt{2r}$ contains all $c(P_i)$, which completes the proof of~\eqref{equation colored upper bound}.

Unfortunately, the proof from~\cite{BarabanshchikovaPolyanskii2024Tverberg} involves an NP-hard problem unless $r=2$, in which case, there is a polynomial algorithm. Hence, the problem of finding a polynomial algorithm proving the tight upper bound~\eqref{equation colored upper bound} remains open.

\subsection{Weak colorful Tverberg problem} 
Since the colorful Tverberg conjecture has not been proven yet, it makes sense to ask the following problem, which seems to have never been thoroughly studied.

\begin{openproblem}
Prove a non-trivial upper bound on $R$ satisfying the following condition. For pairwise disjoint sets $Q_1, \dots, Q_{d+1}$, each consisting of $k$ points in $\mathbb R^d$, there exists a partition ${P_1, \dots, P_k}$ of the union $P := Q_1 \cup \dots \cup Q_{d+1}$, where $|P_i \cap Q_j| = 1$ for each $i\in[k]$ and $j\in[d+1]$, and a ball of radius $R \cdot \diam P$ intersecting the convex hulls of $P_i$ for each $i\in[k]$.
\end{openproblem}

\section{No-dimensional problems on diametral balls}
\label{section k=2}

Recall that the \textit{diametral ball} for two points in $\ell_2$ is defined as the ball whose center coincides with their midpoint and whose radius is half the distance between them. In other words, the pair of points is antipodal for their diametral ball. We denote the diametral ball for points $a,b$ by $B(ab)$.

\subsection{Red-blue Tverberg matchings} 
\label{subsection red-blue tverberg matchings}
For the sake of brevity, for an even set of points in $\ell_2$, a matching is called a \textit{Tverberg matching} if the diametral balls induced by its edges share a common point. Recently, Huemer \textit{et al.} studied the following no-dimensional problem.

\begin{problem}
\label{problem red-blue no-dimensional}
    Is it true that for any $k$ red and $k$ blue points in $\ell_2$, there is a perfect red-blue Tverberg matching?
\end{problem}

As discussed earlier, an affirmative answer to \cref{problem red-blue no-dimensional} completely resolves \cref{problem colorful euclidean} when $r=2$. Indeed, for $k$ red and $k$ blue points, if there exists a perfect red-blue matching $\{e_1, \dots, e_k\}$ such that the diametral balls induced by its edges share a common point, then this common point is located at a distance of at most $\diam e_i/2$ from the centroid $c(e_i)$ for each $i\in[k]$. Consequently, this implies the tight bound $R_c'(\ell_2, k,2) \leq 1/2$.

Huemer \textit{et al.}~\cite{huemer2019matching} solved \cref{problem red-blue no-dimensional} under the assumption that all points lie in the plane. Moreover, they showed that a red-blue matching maximizing the sum of the squared distances between the matched points works well in this problem. Unfortunately, their proof heavily relies on analyzing arrangements of points in the plane. 

In 2024, Pirahmad, Vasilevskii, and the author~\cite{Pirahmad2022}*{Theorem~1.4} generalized this result to $\ell_2$ and their proof avoids any case analysis. We include a sketch of their proof.

\begin{proof}[Sketch of the solution to \cref{problem red-blue no-dimensional}]
    Given $k$ red points and $k$ blue points in $\ell_2$, consider all possible red-blue matchings. Among these matchings, let $\mathcal M_1$ be one that maximizes the function $f$, which depends on a red-blue matching $\mathcal M$ and is defined by
    \[
    f(\mathcal M)=\sum_{rb\in \mathcal M} \|r-b\|^2.
    \]
    To prove that $\mathcal M_1$ satisfies the conclusion of the theorem, consider another function $g:\ell_2 \to \mathbb R$ defined by
    \[
    g(x)=\max_{rb\in \mathcal M_1} \{ \langle r-x,b-x \rangle \}= \max_{rb\in \mathcal M_1}\Big\{
    \Big\|\frac{r+b}{2}-x\Big\|^2-\Big\|\frac{r-b}{2}\Big\|^2\Big\}.
    \]
    Note that if $\langle r-x,b-x\rangle\leq 0$, then $x$ belongs to the ball $B(rb)$. Hence, if $g$ attains a non-positive value at some point, then this point is common for the balls $B(rb)$, $rb\in \mathcal M_1$.

    Without loss of generality assume that $g$ attains its minimum at the origin $o$ and $g(o)>0$. Consider a submatching $\mathcal M_2$ defined by
    \[
        \mathcal M_2=\big\{rb\in \mathcal M_1:\langle r, b\rangle =g(o)\big\}.
    \]
    By considering the subdifferential of $g$ at $o$, one can easily show that
    \[
    o\in \conv \Big\{ \frac{r+b}{2}:rb\in \mathcal M_2 \Big\}.
    \]

    So, there are $r_1b_1, \dots, r_nb_n\in \mathcal M_2$ and positive $\lambda_1,\dots, \lambda_n$ such that
    \[
        o=\sum_{i=1}^n \lambda_i \frac{r_i+b_i}{2} \text{ or equivalently } \sum_{i=1}^n \lambda_i r_i = - \sum_{i=1}^n \lambda_i b_i.
    \]

    Therefore,
    \[
    0\geq \langle \sum_{i=1}^n \lambda_i r_i ,\sum_{i=1}^n \lambda_i b_i\rangle =\sum_{i=1}^n \lambda_i^2  \langle r_i ,b_i\rangle + \sum_{1\leq i < j \leq n} \lambda_i\lambda_j (\langle r_i, b_j\rangle +\langle r_j, b_i \rangle).
    \]
    Since $\langle r_i,b_i\rangle =g(o)>0$, we conclude that there exist distinct $i,j$ such that 
    \[ 
    \langle r_i, b_j\rangle + \langle r_j, b_i\rangle<0<2g(o)=\langle r_i,b_i\rangle +\langle r_j, b_j\rangle.\]
    Hence 
    \[
    \|r_i-b_j\|^2+\|r_j-b_i\|^2>\|r_i-b_i\|^2+\|r_j-b_j\|^2, 
    \]
    and thus for the matching $\mathcal M_3=\mathcal M_1\setminus \{r_ib_i, r_jb_j\} \cup \{r_ib_j, r_jb_i\}$, we have $f(\mathcal M_3)>f(\mathcal M_1),$ a contradiction with the choice of the matching $\mathcal M_1$.
\end{proof}

The solution~\cite{BarabanshchikovaPolyanskii2024Tverberg} to \cref{problem colorful euclidean new} for more than 2 colors follows the same strategy. Interestingly, the proofs in~\cites{Pirahmad2022, BarabanshchikovaPolyanskii2024Tverberg} involve two functions, unlike the proofs~\cites{tverberg1981generalization, Tverberg1993, Roudneff2001} of Tverberg's theorem, which used only one function that depends on a partition and a potential common point. In the solution to \cref{problem red-blue no-dimensional}, the function $f$ is needed to choose a desired matching (partition), while the function $g$ indicates whether a point belongs to all balls induced by a matching. We refer the interested reader to Section 8 in~\cite{Pirahmad2022}, where the authors explain why attempting to mimic the proofs in~\cites{tverberg1981generalization,Tverberg1993, Roudneff2001} (that is, using only the function 
$g$ without $f$) does not yield an affirmative answer to \cref{problem red-blue no-dimensional}.

We note that Pirahmad, Vasilevskii, and the author~\cite{Pirahmad2022}*{Theorem~1.2} gave an alternative solution to \cref{problem red-blue no-dimensional} in the plane, using a Borsuk--Ulam-type argument. This observation naturally raises the question of whether topological methods could be effectively applied to \cref{problem red-blue no-dimensional} in the general case and other no-dimensional Tverberg-type problems.

\subsection{Max-sum matchings}
As we can see from the sketch of the solution to \cref{problem red-blue no-dimensional}, the key idea is to choose a matching that maximizes some function. It turns out that this approach is quite fruitful in the context of various generalizations of \cref{problem red-blue no-dimensional}. As an example, in addition to the results in this subsection, see also the discussion on max-sum trees in the next one.

For any even set of points in $\ell_2$, a matching that maximizes the sum of the distances between the matched points is called a \textit{max-sum matching}. Inspired by the result of Huemer \textit{et al.}~\cite{huemer2019matching}, Bereg \textit{et al.}~\cite{bereg2019maximum} considered the problem of whether a max-sum matching in the plane is a Tverberg matching.

\begin{openproblem}
    \label{question maximizing matching}
    Is it true for any even set of points in $\ell_2$, its max-sum matching is a Tverberg matching?
\end{openproblem}

Bereg \textit{et al.}~\cite{bereg2019maximum} gave an affirmative answer to \cref{question maximizing matching} under the assumption that all points lie on the plane. An alternative proof of a slightly stronger form of this result was provided by Barabanshchikova and the author~\cite{barabanshchikova2024intersectingballs}*{Theorem~2}. Both proofs are based on a routine case analysis in the plane. Unfortunately, a higher-dimensional generalization of this result remains open; see~\cite{Pirahmad2022}*{Problem~9.3}.

The motivation of Bereg~\textit{et al.}~\cite{bereg2019maximum} for studying max-sum matchings comes from a conjecture of Fingerhut~\cite{Eppstein1995}. It was shared with Eppstein who added it to his collection of research problems in computational geometry.

\begin{problem}[Fingerhut's conjecture]
Let a matching $\{a_i b_i : i\in[n]\}$ be a max-sum matching for a set of $2n$ points in the plane. Prove that there exists a point $o$ in the plane such that
\begin{equation}
\label{equation fingerhut}
    \|a_i - o\|+\|b_i - o\| \leq \frac{2}{\sqrt{3}} \|a_i - b_i\| \text{ for all }i\in[n].
\end{equation}
\end{problem}

By considering the vertex set of an equilateral triangle taken twice (that is, a multiset of 6 points in the plane), one can easily compute that the multiplicative constant $2/\sqrt{3}$ in the problem is chosen optimally. Fingerhut's conjecture can be considered a Tverberg-type problem concerning the intersection property of ellipses induced by the max-sum matching. The result of Bereg \textit{et al.} implies that this conjecture holds with the constant $2/\sqrt{3}$ replaced by $\sqrt{2}$. Barabanshchikova and the author~\cite{barabanshchikova2024intersectingellipses} completely resolved this conjecture using an optimization approach similar to the one from the sketch of the solution to \cref{problem red-blue no-dimensional}.

Very recently, the colorful version of Fingerhut's conjecture was proved by~P{\'e}rez-Lantero and Seara~\cite{perez2024center} based on ideas from~\cite{barabanshchikova2024intersectingellipses}. Their result corresponds to the planar case of the following open no-dimensional problem.

\begin{openproblem}[No-dimensional colorful Fingerhut problem]
\label{question fingerhut colorful}
    Let a matching $\{r_i b_i : i = 1,\dots, n\}$ be a red-blue max-sum matching for a set of $n$ red and $n$ blue points in $\ell_2$. Is it true that there exists a point $o$ in $\ell_2$ such that
\[
    \|r_i - o\|+\|b_i - o\| \leq \sqrt{2} \|r_i - b_i\| \text{ for all }i\in[n]?
\]
\end{openproblem}

By considering \cref{example simplex} for $r=2$, one can easily verify that the multiplicative constant $\sqrt2$ in \cref{question fingerhut colorful} is the best possible.

Interestingly, Fekete and Meijer~\cite{fekete1997angle} proved the existence of a perfect Tverberg matching in the plane back in 1997. Using a Borsuk--Ulam-type argument, they showed that for any even set of points in the plane, there exists a perfect matching \( \mathcal{M} \) and a point \( z \) in the plane such that either \( z \in \{x,y\} \) or \( \angle xzy \geq 2\pi/3 \) for all \( xy\in \mathcal{M}\). In particular, this implies that for any even set of points, there is a perfect matching satisfying~\eqref{equation fingerhut} from Fingerhut's conjecture. Essentially, the same fact, with the same proof, was also found by Dumitrescu, Pach, and T\'oth in the paper\cite{dumitrescu2012drawing}*{Lemma 2} from 2012.

\subsection{Tverberg graphs} Sober\'on and Tang~\cite{soberon2020tverberg}*{Problem~1.1} considered a notable generalization of \cref{red-blue matching} about the so-called Tverberg graphs. Let \( G \) be a graph whose vertex set is a finite set of points in \( \ell_2 \). We say that \( G \) is a \textit{Tverberg graph} if
\[
    \bigcap_{ xy\in E(G)} B (xy) \neq \emptyset.
\]

\begin{openproblem}
\label{problem Tverberg graphs global}
    For a finite set of points in $\ell_2$, determine the family of its Tverberg graphs.
\end{openproblem}

As a particular problem, Sober\'on and Tang~\cite{soberon2020tverberg} suggested studying the existence of Tverberg graphs with special properties such as being a perfect matching, Hamiltonian cycle, spanning tree, etc.
For example, they proved that for an odd point set in the plane, there exists a Hamiltonian cycle which is a Tverberg graph. Their proof is based on an optimization approach that differs from the method used in the solution to \cref{problem red-blue no-dimensional}. Later, Pirahmad, Vasilevskii, and the author~\cite{Pirahmad2022}*{Theorem~1.1} found a simple Borsuk--Ulam-type proof showing the same for any finite point set in the plane. A higher-dimensional version of this problem remains open; see~\cite{Pirahmad2022}*{Problem~9.1}.

\begin{openproblem}
\label{question Hamiltonian cycles}
    Is it true for any finite set of points in $\ell_2$, there is a Hamiltonian cycle that is a Tverberg graph?
\end{openproblem}

In 2023, Abu-Affash \textit{et al.}~\cite{Affash2022piercing} showed that for any point set in the plane, a spanning tree that maximizes the sum of distances between connected points, the so-called \textit{max-sum tree}, is a Tverberg graph. Their proof, like many other planar results mentioned above, relies on case analysis, which is specific for the plane. The key observation is that the center of the smallest radius ball enclosing the point set belongs to all balls induced by the edges of the max-sum tree. Later, using this observation and connectivity of non-acute graphs (see also \cite{Pirahmad2022}*{Section~5.2}), Barabanshchikova and the author~\cite{barabanshchikova2024intersectingballs}*{Theorem~1} extended the result on max-sum trees in $\ell_2$.

\smallskip

We conclude this subsection with the following Tur\'an-type problem (compare this with \cref{no-dimensional selection lemma}).
\begin{openproblem}
    For a positive integer $n$, find the maximum $m:=m(n)$ such that any $n$-point set in $\ell_2$ has a Tverberg graph with $m$ edges.
\end{openproblem}
We conjecture that \( m(n) \) has quadratic asymptotics in \( n \). If true, it would be interesting to determine the coefficient of \( n^2 \).

\section{No-dimensional Tverberg problems in Banach spaces}
\label{section banach}

In 2021, Ivanov~\cite{Ivanov2021} initiated the study of the no-dimensional colorful Tverberg problem when the point set $P$ lies in an arbitrary Banach space. This problem remains far from being resolved in its generality.

\begin{openproblem}[colorful no-dimensional Tverberg problem for Banach spaces]
\label{problem colorful banach}
Given a Banach space $X$, integers $k\geq 2$ and $r\geq 1$, find the minimum $R_c:=R_c(X, k,r)$ such that for any pairwise disjoint sets $Q_1,\dots,Q_r$ in $X$, each of size $k$, there is a partition $\{P_1,\dots, P_k\}$ of the union $P=Q_1\cup \dots \cup Q_r$, with $|P_i\cap Q_j|=1$ for any $i\in[k]$ and $j\in[r]$, and a ball of radius $R_c \cdot \max_{j\in [r]} \diam Q_j$ intersecting $\conv P_i$ for all $i\in[k]$.
\end{openproblem}

\subsection{Lower bounds}
Our first goal is to prove a lower bound for $R_c(X, k,r)$ that holds for any infinite-dimensional Banach space $X$.
\begin{proposition}
\label{proposition: banach infinite}
    We have $R_c(\ell_2,k,r)\leq 2 R_c(X,k,r)$.
\end{proposition}
This proposition is a direct corollary of Dvoretzky's theorem~\cite{albiac2006topics}*{Theorem~12.3.6} and the following lemma.
\begin{lemma}
\label{lemma subspace}
    If $Y$ is a subspace of a Banach space $X$, then
        $R_c(Y,k,r) \leq 2R_c(X,k,r).$
\end{lemma}
\begin{proof}[Proof of~\cref{proposition: banach infinite}]
 Since any $kr$ points in $\ell_2$ lie in a linear subspace of dimension $kr$, we have 
    \begin{equation}
    \label{equation choise of d}
        R_c(\ell_2,k,r)=R_c(\mathbb R^{kr}, k,r).
        \end{equation}

    For $d=kr$ and any $\varepsilon>0$, Dvoretzky's theorem claims that there is a $d$-dimensional subspace $X_\varepsilon$ of $X$ and a linear map $T:\mathbb R^d \to X_\varepsilon$ such that 
    \[
        \|x\|_2\leq \|T(x)\|_X\leq (1+\varepsilon)\|x\|_2
    \]
    for any $x\in X$. Using the image of~\cref{example simplex} under $T$, we obtain a set in $X_\varepsilon$ showing that 
    \[ \frac{ R_c(\mathbb R^d, k,r)}{1+\varepsilon} \leq R_c(X_\varepsilon,k,r).\]
    Using \cref{lemma subspace}, this inequality, and \eqref{equation choise of d}, we conclude
    \[\frac{R_c(\ell_2,k,r)}{1+\varepsilon}= \frac{R_c(\mathbb R^d, k,r)}{1+\varepsilon}\leq R_c(X_\varepsilon,k,r)\leq 2R_c(X,k,r).\]
    Letting $\varepsilon\to 0$ completes the proof of \cref{proposition: banach infinite}.
\end{proof}
\begin{proof}[Proof of \cref{lemma subspace}]
To prove this lemma, it is enough to show that the intersection of any unit ball \( B \) in the Banach space \( X \) and the Banach space \( Y \subset X \) can be covered by a ball of radius 2 centered in the Banach space \( Y \). Indeed, by the triangle inequality, the intersection \( Y \cap B \) has a diameter of at most 2. Thus, any ball of radius 2 centered at a point in \( Y \cap B \) covers this set.\end{proof}

\cref{proposition: banach infinite} leaves the following question open.

\begin{openproblem}
\label{openproblem lower bound}
    Is it true that for any infinite-dimensional Banach space $X$, we have
    \[
        R_c(\ell_2,k,r)\leq R_c(X,k,r)?
    \]
\end{openproblem}

\begin{example}
\label{example simplex 2}
By using \cref{example simplex} for $k=2$, one can easily obtain the lower bound 
\[
    2^{-1}r^{\frac{1-p}{p}}\leq R_c(\ell_p,2,r), \text{ where } 1\leq p\leq \infty
\]
\end{example}
\begin{example}
\label{example l^1}
Ivanov~\cite{Ivanov2021}*{the last paragraph of Section~1} proposed an example showing that 
\begin{equation}
\label{equation l1 lower bound}
R_c(\ell_1, k, r) \geq \frac{1}{2}.% \text{ and }R_c'(\ell_1, k, r) \geq \frac{1}{2}.
\end{equation}
Consider \(r\) sets \(Q_i := \{e_{ik+1}, \dots, e_{(i+1)k}\}\), where \(i \in [r]\) and \(e_j\) is the \(j\)-th basis vector of \(\ell_1\). Suppose that for a partition \(\{P_1, \dots, P_k\}\) of the union \(Q_1\cup\dots \cup Q_r\), with $|P_i\cap Q_j|=1$ for any $i\in[k]$ and $j\in [r]$, there is a ball intersecting every \(\conv P_i\), $i\in [k]$. Since $\max_{j\in[r]}\diam Q_j=2$, we aim to prove~\eqref{equation l1 lower bound}, it is enough to show that the radius of this ball is at least~1.

Let \(a\) and \(b\) be any two points lying in distinct convex hulls \(\conv P_i\). A simple computation shows that the distance between \(a\) and \(b\) in \(\ell_1\) equals 2. Assuming that \(a\) and \(b\) lie in the ball, and using the triangle inequality, we conclude that \(2 = \|a-b\|_1\) does not exceed the diameter of this ball. This completes the argument establishing~\eqref{equation l1 lower bound}.

\end{example}

\begin{example} 
\label{example l^infty}
It is well-known that there is an isometric embedding $T:\ell_1\to \ell_\infty$, that is, $\|T(x)-T(y)\|_\infty=\|x-y\|_1$ for any $x,y\in \ell_1$; see~\cite{albiac2006topics}*{Theorem 2.5.7}. Using the image of \cref{example l^1} under $T$ and repeating the above argument, one can easily conclude that
\[
    R_c(\ell_\infty, k,r)\geq \frac12.
\]
For details, we refer to \cite{BarabanshchikovaPolyanskii2024Tverberg}*{Section 2.2}.
\end{example}

The following open problem arises from Examples~\ref{example l^1} and~\ref{example l^infty}.
\begin{openproblem}
    \label{question l_p}
    Is it true that 
    \[R_c(\ell_1, k,r)=R_c(\ell_\infty, k,r)=\frac12? \]
\end{openproblem}

\subsection{Upper bounds.} 
Bohnenblust~\cite{bohnenblust1938convex} proved a generalization of Jung's theorem for arbitrary Banach space, which implies that $R_c(X,k,1)\leq \frac{k-1}{k}$; see also \cref{lemma new} and~\cites{leichtweiss1955zwei,dol1987jung}, where the equality cases were studied. Pichugov~\cite{pichugov1990jung} studied the problem of Jung for $\ell_p$ and proved 
\[
R_c(\ell_p, k,1)\leq 2^{-1/s} \left(\frac{k-1}{k}\right)^{1/p} \text{ for }1\leq p<\infty, \text{ where }s=\max\{p, p(p-1)^{-1}\}. 
\]

In his paper~\cite{Ivanov2021}, Ivanov obtained upper bounds for Banach spaces of type $p$ with $p>1$. A Banach space $X$ is called to be \textit{of type $p$} for some $p\geq 1$ if there exists a constant $T_p(X)$ such that, for every finite set $S$ of vectors in $X$, the average norm of the $2^{|S|}$ vectors of the form $\sum_{s\in S} \pm s$ does not exceed
\[
T_p(X)\Big( \sum_{s\in S} \|s\|^p \Big)^{1/p};
\]
see also \cite{lindenstrauss2013classical}*{Definition~1.e.12}. By the triangle inequality, any Banach space is of type 1. It is well-known that the type of a Banach space cannot exceed 2, and the maximum possible type of $\ell_p$, $p\geq 1$, is $\min\{2,p\}$. The maximum possible type of $\ell_\infty$ is 1.

For any Banach space $X$ of type $p$ with $p>1$,
Ivanov proved that
\[
R_c(X, k,r)\leq C(X) r^{\frac{1-p}{p}},
\]
where a constant $C(X)>0$ depends on $X$. For that, he adapted an averaging technique by Adiprasito \textit{et al.}~\cite{adiprasito2020theorems} to his problem and followed their strategy
to complete his proof.

It is worth mentioning that for any Banach space $X$ (including spaces that are not of type $p$ for any $p>1$ but of type 1), $R_c(X, k,r)$ is bounded from above by a constant. The following lemma generalizes the result of Bohnenblust~\cite{bohnenblust1938convex} for $r=1$.

\begin{lemma}
\label{lemma new} For any Banach space $X$ and integers $k\geq 2$ and $r\geq 1$, we have 
\[
    R_c(X,k,r)\leq \frac{k-1}{k}.
\]
\end{lemma}
\begin{proof} Let $Q_1,\dots, Q_r$ be pairwise disjoint point sets in $X$, each of size $k$. We show that any partition $\{P_1,\dots, P_k\}$ of $Q_1\cup \dots \cup Q_r$ into $k$ parts, with $|P_i\cap Q_j|=1$ for any $i\in[k]$ and $j\in [r]$, can be used to complete the proof. 

Indeed, for any $q_j \in Q_j$, $j\in[r]$, we have
\[
\|q_j - c(Q_j)\| = \frac{1}{k}\left\| \sum_{q \in Q_j \setminus \{q_j\}} (q_j - q)\right\|
\leq \frac{1}{k} \sum_{q \in Q_j \setminus \{q_j\}} \|q_j - q\|
\leq \frac{k-1}{k} \diam Q_j.
\]

Consider any $P_i=\{q_1,\dots, q_r\}$ with $q_j\in Q_j$, $j\in[r]$. Then
\[
\left\| \frac{q_1 + \cdots + q_r}{r} - c(P) \right\|
= \left\| \frac{q_1 + \dots + q_r}{r} - \frac{c(Q_1) + \dots + c(Q_r)}{r} \right\|
\leq \frac{1}{r} \sum_{i \in [r]} \|q_i - c(Q_i)\|.
\]
Using the previous inequality here, we complete the proof of the first inequality. 
\end{proof}

The following problem is inspired by \cref{lemma new}.

\begin{openproblem}
\label{openproblem upper bound}
Which Banach space $X$ has the largest value of $R_c(X, k, r)$?
\end{openproblem}

\sloppy
\begin{remark}
\label{remark2}
Some generalizations of the no-dimensional Carath\'eodory's theorem, \cref{caratheodory}, were studied in~\cites{combettes2023revisiting,ivanov2021approximate} in the context of algorithmic aspect. It would be interesting to understand whether the algorithms from these papers could lead to an algorithmic solution to \cref{problem colorful banach}.
\end{remark}

\subsection{Second no-dimensional colored Tverberg problem}

Similarly to the Euclidean case, one can consider a colorful variation of \cref{problem colorful banach}, by replacing $\max_{j \in [r]} \diam Q_j$ with 
\[
\diam(Q_1,\dots, Q_r):=\max \{\|x-y\|_X \mid x\in Q_i,y\in Q_j \text{ with } i\ne j\};
\]
see~\cite{BarabanshchikovaPolyanskii2024Tverberg}.

\begin{problem}[second colorful no-dimensional Tverberg problem for Banach spaces]
\label{problem colorful banach new}
Given a Banach space $X$, integers $k\geq 2$ and $r\geq 2$, find the smallest value of $R_c':=R_c'(X, k,r)$ such that for pairwise disjoint sets $Q_1,\dots, Q_r$ in $X$, each of size $k$, there is a partition $\{P_1,\dots, P_k\}$ of the union $P=Q_1\cup\dots \cup Q_r$, with $|P_i\cap Q_j|=1$ for any $i\in [k]$ and $j\in [r]$, and a ball of radius $R'_c \cdot \diam(Q_1,\dots, Q_r)$ intersecting $\conv P_i$ for all $i\in [k]$.
\end{problem}

Note that \cref{proposition: banach infinite}, Lemmas \ref{lemma subspace} and \ref{lemma new}, as well as Examples \ref{example simplex 2}, \ref{example l^1}, and \ref{example l^infty}, remain valid when $R_c(X,k,r)$ is replaced with $R'_c(X,k,r)$. (For \cref{example simplex 2}, however, the lower bound is slightly modified: $2^{-1/p} r^{\frac{1-p}{p}} \leq R_c'(\ell_p,2,r)$.) In fact, all the proofs are essentially the same as those given above.

Very recently, Barabanshchikova and the author~\cite{BarabanshchikovaPolyanskii2024Tverberg} showed that the Banach space $\ell_\infty$, which is not of type $p$ for any $p>1$, but is of type 1, satisfies the tight upper bound
\[
R_c'(\ell_\infty, k,r)\leq \frac{1}{2}.
\]
Their proof strategy is very similar to the idea explained at the beginning of \cref{section k=2}: The balls centered at some points in the convex hulls of $P_i$ of radius $R$ have a common point. This verification for $\ell_\infty$ is simple because a collection of balls in $\ell_\infty$ has a common point if and only if any two balls have a common point.

It is interesting to study Open Problems~\ref{openproblem lower bound}, \ref{question l_p}, and~\ref{openproblem upper bound}, with $R_c(X,k,r)$ replaced by $R_c'(X,k,r)$.

\subsection{Tverberg graphs in Banach spaces}
It is interesting to study generalizations of Tverberg graphs for Banach spaces; see~\cite{Pirahmad2022}*{Problem~9.3}. Let $K$ be the unit ball in a Banach space centered at the origin, that is, $K$ is a centrally symmetric about the origin convex body. For two points $x, y \in \mathbb{R}^d$, let $K(xy) = (x + y)/2 + \lambda K$, where $\lambda$ is the least positive number such
that $(x + y)/2+\lambda K$ covers the points $x$ and $y$, that is, the line segment $xy$ is a diameter of
the ball $K(xy)$. Let $G$ be a graph whose vertex set is a finite set of points in $\mathbb{R}^d$. We say that $G$ is a \textit{$K$-Tverberg graph} if
\[
\bigcap_{xy\in E(G)} K(xy) \ne \emptyset.
\]

\begin{openproblem}
Let $K$ be the unit ball centered at the origin in a Banach space. Is it true that for any even set of points, there exists a perfect matching that is a $K$-Tverberg graph? Is it true that the matching that maximizes the sum of the distances between matched points is a $K$-Tverberg graph?
\end{openproblem}

As a starting point to study this problem, one can consider the planar case.

\section{Applications}
\label{section: applications}

As with any fundamental result in mathematics, Tverberg's theorem has numerous applications. The most well-known applications include the centerpoint theorem, the selection lemma, and the existence of weak $\varepsilon$-nets. These results are among the fundamental ones in computational geometry. They describe combinatorial convexity properties of arrangements of large point sets in $d$-dimensional Euclidean space. For further details, we refer the reader to Matoušek's book~\cite{matousek2013lectures}*{Sections 1.4, 9.1, and 10.4}.

Given the recently obtained no-dimensional versions of Tverberg's theorem, it is natural to establish no-dimensional analogs of the aforementioned applications as well. Adiprasito \textit{et al.}~\cite{adiprasito2020theorems} proved no-dimensional versions of the selection lemma, the weak $\varepsilon$-nets theorem, and the centerpoint theorem. Subsequent improvements to their results on the selection lemma and the weak $\varepsilon$-nets theorem were obtained in~\cite{harpeled2023}. Ivanov~\cite{Ivanov2021} extended these results to Banach spaces of type $p$, where $p>1$. In all these works, the authors applied standard strategies. Below, we present these theorems for Banach spaces.

\begin{theorem}[no-dimensional selection lemma]
\label{no-dimensional selection lemma}
    Let $P$ be an $n$-point set in a Banach space $X$ of type $p>1$. Let $r\in[n]$. Then there is a ball of radius 
    \[
        C(X) r^{\frac{1-p}{p}} \diam P 
    \]
    intersecting the convex hull of $r^{-r}\binom{n}{r}$ $r$-tuples in $P$, where the constant $C(X)$ depends on $p$ and $T_p(X)$.
\end{theorem}

\begin{theorem}[no-dimensional weak $\varepsilon$-net theorem]
    Let $P$ be an $n$-point set in a Banach space $X$ of type $p>1$. Let $r \in [n]$ and $\varepsilon>0$. Then there are at most $r^r\varepsilon^{-r} $ balls each of radius
    \[
        C(X) r^{\frac{1-p}{p}}\diam P
    \]
    in $X$ such that for every $Y \subset P$ with $|Y|\geq \varepsilon n$, the convex hull $\conv Y$ meets at least one of the balls. Here the constant $C(X)$ depends on $p$ and $T_p(X)$.
\end{theorem}

\begin{theorem}[no-dimensional centrepoint theorem]
    Let $P$ be an $n$-point set in a Banach space $X$ of type $p>1$. Let $r\in [n]$. There exists a ball of radius
    \[
        C(X) r^{\frac{1-p}{p}} \diam P
    \]
    in $X$ such that every half-space containing this ball contains at least $n/r$ points form $P$.
\end{theorem}

\section{No-dimensional Tverberg problems in metric spaces} 
\label{section metric spaces}
It is interesting to study a variation of the no-dimensional Tverberg theorem for an abstract metric space $X$. First, let us recall the definition of the convex hull.

A subset $C$ of $X$ is a \textit{geodesically convex set} if for any two points in $C$, there is a unique minimizing geodesic contained within $C$ that joins those two points. The \textit{geodesic convex hull} of $C$ is the intersection of all convex sets containing $C$.

\begin{openproblem}
\label{problem colorful metric}
Given a metric space $X:=(M,d)$, integers $k\geq 2$ and $r\geq 1$, find the minimum $R_c := R_c(X, k, r)$ such that for a pairwise disjoint sets $Q_1, \dots, Q_r$ in $X$, each of size $k$, there is a partition $\{P_1, \dots, P_k\}$ of the union $P = Q_1\cup\dots \cup Q_r$, with $|P_i \cap Q_j| = 1$ for any $i \in [k]$ and $j \in [r]$, and a ball of radius $R_c \cdot \max_{j\in[r]} \diam Q_j$ intersecting the geodesic convex hulls of $P_i$ for all $i\in[k]$.
\end{openproblem}

In particular, it would be interesting to study this problem for Riemannian manifolds. The problem of Jung and, as a corollary, the case $r=1$ in~\cref{problem colorful metric} have been studied for spherical and hyperbolic spaces~\cite{dekster1995jung}, as well as for metric spaces of curvature bounded above~\cite{dekster1997jung}.

Analogously to the Euclidean case, one can consider a colorful variation of \cref{problem colorful metric}, by replacing $\max_{j \in [r]} \diam Q_j$ with $\diam (Q_1,\dots, Q_r)$.

Very recently, Barabanshchikova and the author~\cite{BarabanshchikovaPolyanskii2024Tverberg}*{Theorem 3} mimicked their proof of~\eqref{equation colored upper bound} to show a certain generalization for an infinite-dimensional hyperbolic space $H^\infty$. Unfortunately, it is unclear how to interpret their result in the context of \cref{problem colorful metric}. However, their result implies an affirmative answer to the hyperbolic variation of \cref{problem red-blue no-dimensional} on the existence of a Tverberg matching.

It is worth mentioning that a spherical space does not have such a generalization on Tverberg matchings. Indeed, consider any 4 points on a sphere $S \subset \ell_2$ forming a square in $\ell_2$ and color them alternately in red and blue. A simple verification shows that any red-blue matching for these 4 points induces diametral balls (spherical caps in this case) with an empty intersection.

\section*{Acknowledgments}
This work was partially supported by ERC Advanced Grant ``GeoScape,'' No 882971 and the NSF grant DMS 2349045.

The author is grateful to János Pach for his support and encouragement to write this survey. The author also thanks Imre Bárány for introducing him to his no-dimensional results and Pablo Soberón for introducing him to problems about diametral balls.

The author is grateful to Grigory Ivanov and M\'arton Nasz\'odi for the discussions and valuable remarks. The author also thanks his former students Polina Barabanshchikova, Olimjoni Pirahmad, and Alexey Vasileuski for fruitful collaboration on the problems covered in this survey and for their comments on an earlier version of the manuscript. The author thanks his student Griffin Johnson for his comments on the final version of the manuscript.

The author deeply appreciates the reviewers for their thoughtful reading of the article and numerous suggestions, which significantly improved both the presentation and the overall scope of this work. In particular, the author is indebted to one of the reviewers for their detailed feedback, which, after substantial revision, became \cref{proposition: banach infinite}, \cref{example l^infty}, and \cref{lemma new}.

The author is sincerely grateful to his wife Olga and his son Roman for their constant support and patience, without which completing this survey would not have been possible.

\raggedbottom
\bibliographystyle{alpha}
\bibliography{biblio}

\end{document}